\documentclass[a4paper,11pt]{amsart}
\usepackage{a4wide}
\usepackage[T1]{fontenc}
\usepackage{lmodern}
\usepackage{amsmath,amssymb,mathtools}
\usepackage{url}
\usepackage[colorlinks=true,linkcolor=blue,citecolor=blue,urlcolor=blue]{hyperref}

\newtheorem{maintheorem}{Theorem}

\newtheorem{lemma}{Lemma}[section]
\newtheorem{corollary}[lemma]{Corollary}
\theoremstyle{remark}
\newtheorem{remark}[lemma]{Remark}
\newcommand{\N}{\mathbb N}
\newcommand{\F}{\mathbb F}
\newcommand{\weakto}{\rightharpoonup}

\newcommand{\Rea}{\operatorname{Re}}
\DeclareMathOperator{\aconv}{aconv}

\title{Weakly LUR norms and the Schur property}
\author[S.~Draga]{Szymon Draga}
\address[S.~Draga]{Mathematical Institute\\Czech Academy of Sciences\\\v Zitn\'a 25\\115 67 Praha 1\\Czech Republic}
\email{szymon.draga@gmail.com}
\author[T.~Kania]{Tomasz Kania}
\address[T.~Kania]{Mathematical Institute\\Czech Academy of Sciences\\\v Zitn\'a 25 \\115 67 Praha 1\\Czech Republic  and  Institute of Mathematics and Computer Science\\ Jagiellonian University\\ {\L}ojasiewicza 6, 30-348 Krak\'{o}w, Poland}
\email{kania@math.cas.cz, tomasz.marcin.kania@gmail.com}
\thanks{RVO: 67985840.}
\date{}

\begin{document}
\begin{abstract}
A separable real or complex Banach space fails the Schur property if and
only if it admits an equivalent G\^ateaux smooth, weakly locally uniformly
rotund norm which is not midpoint locally uniformly rotund. The construction
enlarges an LUR unit ball by a weakly compact set and adds a weighted
Hilbert-space term. The resulting norms can be chosen arbitrarily close to
any prescribed equivalent LUR norm, and are G\^ateaux smooth whenever the
prescribed norm is G\^ateaux smooth; Fr\'echet smoothness is preserved as
well. Weakly locally uniformly rotund norms which are not midpoint locally
uniformly rotund are dense among all equivalent norms on each separable
non-Schur space. For arbitrary real or complex Banach spaces, such a
renorming exists exactly when the space is LUR-renormable and fails the
Schur property. We give a direct construction for this last assertion.
\end{abstract}
\maketitle

\section{Introduction}

Local uniform rotundity turns asymptotic equality in the triangle
inequality into norm convergence. Its weak version requires only weak
convergence. The Schur property makes these two conclusions equivalent,
so a space with this property cannot carry a weakly LUR norm which fails
to be LUR. Since LUR implies midpoint local uniform rotundity, the same
obstruction applies to failure of the latter property. The converse asks
whether a weakly null sequence which is not
norm null can always be incorporated into the geometry of an equivalent
norm.

The first-named author \cite[Theorem~1]{Draga} proved the existence of an
equivalent wLUR norm which is not LUR on every infinite-dimensional Banach
space with separable dual, using a bounded shrinking Markushevich basis.
De Bernardi, Preti, and Somaglia \cite[Section~2.1]{DBPS} subsequently
recorded that the same construction already fails midpoint local uniform
rotundity and, under the shrinking hypothesis, is weakly uniformly rotund.
The extension argument in the first-named author's paper also
applies to LUR-renormable spaces containing an infinite-dimensional subspace
with separable dual. The construction below uses a weakly compact set in
place of the shrinking assumption and yields the following characterisation.

\begin{maintheorem}\label{thm:main}
Let $X$ be a separable real or complex Banach space. The following are
equivalent:
\begin{enumerate}
\item $X$ does not have the Schur property;
\item $X$ admits an equivalent weakly locally uniformly rotund norm
which is not locally uniformly rotund;
\item $X$ admits an equivalent G\^ateaux smooth, weakly locally uniformly
rotund norm which is not midpoint locally uniformly rotund.
\end{enumerate}
\end{maintheorem}

Theorem~\ref{thm:main} applies beyond the class covered by the earlier
extension result. Azimi and Hagler \cite[Theorem~1]{AzimiHagler} constructed separable
Banach spaces which fail the Schur property and are hereditarily
$\ell_1$, meaning that every infinite-dimensional closed subspace contains
a copy of $\ell_1$. No infinite-dimensional closed subspace $Y$ of such
a space can have separable dual. Indeed, if $Z\subset Y$ is isomorphic to
$\ell_1$, the Hahn--Banach theorem makes the restriction map
$Y^*\to Z^*\cong\ell_\infty$ surjective, which rules out separability of
$Y^*$.

For infinite-dimensional real spaces with separable dual, De Bernardi,
Preti, and Somaglia
\cite[Theorem~3.4]{DBPS} obtained an equivalent Fr\'echet smooth, weakly
uniformly rotund norm which is not midpoint locally uniformly rotund.
Their result for arbitrary infinite-dimensional separable real spaces
\cite[Theorem~3.3]{DBPS} gives G\^ateaux smoothness, rotundity, and failure
of midpoint local uniform rotundity, without weak local uniform rotundity.
Both results approximate any prescribed equivalent norm. In the
separable-dual setting, their combination of weak uniform rotundity and
Fr\'echet smoothness is stronger than the conclusion of
Theorem~\ref{thm:main}. Our contribution is the simultaneous smooth wLUR
and non-MLUR construction throughout the larger class of separable
non-Schur spaces, including the Azimi--Hagler examples.

Our construction also gives a quantitative approximation result.
\begin{maintheorem}\label{thm:approximation}
Let $X$ be a separable real or complex Banach space without the Schur
property, and let $p$ be an equivalent LUR norm on $X$. For every
$\eta,\delta>0$, there is an equivalent weakly LUR norm $N_{\eta,\delta}$
which is not midpoint locally uniformly rotund and satisfies
\[
 \frac{p(x)}{1+\eta}\leqslant N_{\eta,\delta}(x)
 \leqslant\sqrt{1+\delta^2}\,p(x)\qquad(x\in X).
\]
If $p$ is G\^ateaux smooth, respectively Fr\'echet smooth, then
$N_{\eta,\delta}$ can be chosen with the same smoothness property.
\end{maintheorem}

In particular, these norms approximate $p$ uniformly on $B_p$ as
$\eta,\delta\to0$. A preliminary LUR perturbation gives density among
all equivalent norms; see Corollary~\ref{cor:density}. The norms in
Theorem~\ref{thm:approximation} have the form
\[
 N_{\eta,\delta}(x)^2=q_\eta(x)^2
       +\delta^2\sum_{j=1}^{\infty}2^{-j}|\langle f_j,x\rangle|^2.
\]
Here $(f_j)$ is norming for $p$, and $q_\eta$ is the Minkowski functional
of $B_p+\eta K$, where $K$ is the closed absolutely convex hull of a
normalised weakly null sequence. The weighted Hilbert-space term follows
the approach of the first-named author \cite{Draga}. The additional
ingredient is Lemma~\ref{lem:enlargement}: weak compactness of $K$ and
local uniform rotundity of $p$ turn coordinate convergence on the
$q_\eta$-unit sphere into weak convergence. No norm-density assumption on
the span of $(f_j)$ in $X^*$ is needed.

Molt\'o, Orihuela, Troyanski, and Valdivia
\cite[Main Theorem]{MOTV} proved that every wLUR-renormable space is
LUR-renormable. Together with our construction, this gives the following
complete formulation without a separability assumption.
\begin{maintheorem}\label{thm:general}
For a real or complex Banach space $X$, the following are equivalent:
\begin{enumerate}
\item $X$ admits an equivalent weakly LUR norm which is not MLUR;
\item $X$ admits an equivalent LUR norm and fails the Schur property.
\end{enumerate}
\end{maintheorem}

The sufficiency also follows from Theorem~\ref{thm:main} and the wLUR
extension assertion recorded by the first-named author
\cite[Corollary~6 and Remark~7]{Draga}, in connection with
\cite[Theorem~1.1 and Remark~1.1]{Tang}. We instead prove it directly.
The estimate in Lemma~\ref{lem:restricted-separation} allows the Hilbert-space
operator to separate points only in the separable closed span of $K$.
Thus no norm-extension assertion is needed in our proof.

\section{Preliminaries}

We write $\F$ for $\mathbb R$ or $\mathbb C$. All norms on $X$ considered
below are equivalent to its given Banach-space norm. The pairing
$\langle x^*,x\rangle$ denotes evaluation of $x^*\in X^*$ at $x\in X$.
For an equivalent norm $p$, let
$B_p=\{x:p(x)\leqslant1\}$ and $S_p=\{x:p(x)=1\}$.
All rotundity notions over the complex field are understood on the
underlying real space.
The norm $p$ is \emph{locally uniformly rotund} if
\[
 x_n,x\in S_p,\qquad p(x_n+x)\longrightarrow2
 \quad\Longrightarrow\quad p(x_n-x)\longrightarrow0.
\]
It is \emph{weakly locally uniformly rotund} (wLUR) if the same hypotheses
imply $x_n\weakto x$. It is \emph{weakly uniformly rotund} (WUR) if
$x_n,y_n\in S_p$ and $p(x_n+y_n)\to2$ imply
$x_n-y_n\weakto0$. Thus WUR implies wLUR.
It is \emph{midpoint locally uniformly rotund}
(MLUR) if, whenever $x,x_n^+,x_n^-\in S_p$ and
\[
 p\!\left(\frac{x_n^++x_n^-}{2}-x\right)\longrightarrow0,
\]
we have $p(x_n^+-x_n^-)\to0$.
Every LUR norm is MLUR. Indeed, a supporting functional at $x$ takes real
parts tending to one on both endpoint sequences, so
$p(x_n^\pm+x)\to2$; local uniform rotundity gives
$p(x_n^\pm-x)\to0$.

A norm is \emph{G\^ateaux smooth} if it is G\^ateaux differentiable at
every non-zero point, with differentiation over real parameters in the
complex case. Equivalently, for each $x\in S_p$ the set
\[
 D_p(x)=\{x^*\in X^*: \|x^*\|_{p^*}=1,\ \langle x^*,x\rangle=1\}
\]
is a singleton; see \cite{DGZ}.
A norm is \emph{Fr\'echet smooth} if it is Fr\'echet differentiable at
every non-zero point. We use \v Smulyan's criterion in the following
form: $p$ is Fr\'echet differentiable at $x\in S_p$ if and only if every
sequence $(g_n)\subset S_{p^*}$ with
$\Rea\langle g_n,x\rangle\to1$ converges in dual norm to the unique
element of $D_p(x)$; see \cite{DGZ}.
A Banach space has the \emph{Schur property} if every
weakly convergent sequence converges in norm. Equivalent norms have the
same continuous linear functionals and hence the same weak convergence.

We use the classical simultaneous renorming theorem: every separable real
Banach space admits an equivalent norm which is both LUR and G\^ateaux
smooth. This follows from the LUR renorming theorem of Kadec
\cite{Kadec} and the simultaneous renorming method of Asplund
\cite{Asplund}; see also \cite{DGZ} and
\cite[Theorem~3.1]{DBS}. The following observation supplies the complex
case and will also be used in the proof of Theorem~\ref{thm:general}.

\begin{lemma}\label{lem:rotation}
Let $X$ be a complex Banach space and let $r$ be an equivalent norm on
its underlying real space. Define
\begin{equation}\label{eq:rotation}
 p(x)^2=\int_0^{2\pi}r(e^{it}x)^2\,\frac{dt}{2\pi}.
\end{equation}
Then $p$ is an equivalent complex norm. If $r$ is LUR, respectively
wLUR, then $p$ has the same property. If $r$ is G\^ateaux smooth, then
$p$ is G\^ateaux smooth. If $Y$ is a complex subspace and $r|_Y$ is a
complex norm, then $p|_Y=r|_Y$.
\end{lemma}
\begin{proof}
The triangle inequality follows from Minkowski's integral inequality,
and invariance under multiplication by complex scalars of modulus one
follows by translation of the integration variable. The other norm
axioms, equivalence, and the assertion about restriction are immediate.

Suppose that $r$ is LUR or wLUR, that $p(x_n)=p(x)=1$, and that
$p(x_n+x)\to2$. Put
$a_n(t)=r(e^{it}x_n)$ and $a(t)=r(e^{it}x)$. The inequalities
\[
 p(x_n+x)\leqslant\|a_n+a\|_{L_2}\leqslant2
\]
and the parallelogram identity give $a_n\to a$ in $L_2$. Also the
non-negative functions
\[
 (a_n+a)^2-r(e^{it}(x_n+x))^2
\]
have integrals tending to zero. From any subsequence we may therefore
extract one on which both convergences hold pointwise almost everywhere.
At any such $t$, we have
$r(e^{it}x_n)\to r(e^{it}x)>0$ and
$r(e^{it}(x_n+x))\to2r(e^{it}x)$.
Normalising and applying the corresponding property of $r$ gives
$e^{it}x_n\to e^{it}x$ in norm, respectively weakly. Multiplication by
$e^{-it}$ is a bounded real-linear operator, and the real and complex
weak topologies coincide. The subsequence principle therefore gives the
required convergence of $x_n$ to $x$.

Suppose now that $r$ is G\^ateaux smooth. Fix $x\neq0$ and $h\in X$.
For real $s$ with $0<|s|\leqslant1$, the absolute value of the difference
quotient of $r(e^{it}(x+sh))^2$ at $s=0$ is bounded by
\[
 \bigl(2r(e^{it}x)+r(e^{it}h)\bigr)r(e^{it}h).
\]
This expression is bounded uniformly in $t$. Pointwise derivatives exist
and are real-linear in $h$, so dominated convergence differentiates
\eqref{eq:rotation} under the integral. The resulting derivative is
bounded and real-linear. Thus $p^2$, and hence $p$ away from zero, is
G\^ateaux differentiable.
\end{proof}

We shall also use the following elementary consequence of the definition.
If $p$ is LUR, $p(y)=1$, $p(y_n)\leqslant1$, and
$p(y_n+y)\to2$, then $p(y_n-y)\to0$. Indeed,
$p(y_n)\to1$ by the triangle inequality, after which one applies the
definition to $y_n/p(y_n)$.

If $(X,p)$ is separable and non-zero, there is a sequence
$(f_j)_{j\in\N}\subset X^*$ such that
\begin{equation}\label{eq:norming}
 \|f_j\|_{p^*}=1\quad(j\in\N),
 \qquad p(x)=\sup_{j\in\N}|\langle f_j,x\rangle|\quad(x\in X).
\end{equation}
Indeed, choose a dense sequence $(d_j)$ in $S_p$ and, by the Hahn--Banach theorem,
choose $f_j$ with $\|f_j\|_{p^*}=1$ and $\langle f_j,d_j\rangle=1$.
For $x\in S_p$, the inequality
$|\langle f_j,x\rangle|\geqslant1-p(x-d_j)$ and density give
$\sup_j|\langle f_j,x\rangle|=1$. Homogeneity proves \eqref{eq:norming}.

\begin{lemma}\label{lem:compact}
Suppose that $z_n\weakto0$ and $p(z_n)\leqslant1$. Then
\[
 K=\left\{\sum_{n=1}^{\infty}a_nz_n:
               a\in\ell_1(\F),\ \sum_{n=1}^{\infty}|a_n|\leqslant1\right\}
\]
is weakly compact and equals
$\overline{\aconv}\{z_n:n\in\N\}$, where the closure is in norm.
In particular, $K\subset B_p$.
\end{lemma}
\begin{proof}
The series defining $K$ converge absolutely in $X$. Define
$A:\ell_1(\F)\to X$ by $Aa=\sum_n a_nz_n$. For every $g\in X^*$,
the sequence $(\langle g,z_n\rangle)$ belongs to $c_0(\F)$, and
\[
 \langle g,Aa\rangle=\sum_{n=1}^{\infty}a_n\langle g,z_n\rangle.
\]
Thus $A$ is continuous from $\sigma(\ell_1,c_0)$ to $\sigma(X,X^*)$.
The Banach--Alaoglu theorem makes $B_{\ell_1}$ weak-star compact, so
$K=A(B_{\ell_1})$ is weakly compact. It is absolutely convex, contains
every $z_n$, and is norm closed. Conversely, truncating the defining
series approximates every element of $K$ in norm by elements of
$\aconv\{z_n:n\in\N\}$. Finally, $p(Aa)\leqslant\|a\|_1$.
\end{proof}

\section{A weakly compact enlargement of an LUR ball}

\begin{lemma}\label{lem:enlargement}
Let $p$ be an equivalent LUR norm on $X$, let $(f_j)$ satisfy
\eqref{eq:norming}, let $K\subset B_p$ be non-empty, absolutely
convex and weakly compact, and let $\eta>0$. The set
\[
 B=B_p+\eta K
\]
is the closed unit ball of an equivalent norm $q$, and
\begin{equation}\label{eq:equiv}
 \frac{p(x)}{1+\eta}\leqslant q(x)\leqslant p(x)\qquad(x\in X).
\end{equation}
Moreover, if
\begin{equation}\label{eq:coordinates}
 q(x_n)\longrightarrow q(x),\qquad
 \langle f_j,x_n\rangle\longrightarrow\langle f_j,x\rangle\quad(j\in\N),
\end{equation}
then $x_n\weakto x$. If $p$ is G\^ateaux smooth, respectively
Fr\'echet smooth, then $q$ has the same smoothness property.
\end{lemma}
\begin{proof}
The set $B$ is absolutely convex and
$B_p\subset B\subset(1+\eta)B_p$. It is norm closed. Indeed, suppose
$b_n+\eta k_n\to v$ in norm, with $b_n\in B_p$ and $k_n\in K$.
Weak compactness and the Eberlein--\v Smulian theorem give a subsequence
such that $k_n\weakto k\in K$. Along it,
$b_n\weakto v-\eta k$, so $v-\eta k\in B_p$ because $B_p$ is weakly closed.
Hence $v\in B$. Its Minkowski functional
$q(x)=\inf\{t>0:x\in tB\}$ is therefore an equivalent norm with closed
unit ball $B$, and the
inclusions imply \eqref{eq:equiv}.

We next prove the assertion about weak convergence when $q(x_n)=q(x)=1$.
Write
\[
 x_n=b_n+\eta k_n,\qquad b_n\in B_p,\quad k_n\in K.
\]
Start with an arbitrary subsequence. Pass to a further subsequence,
without changing notation, such that $k_n\weakto k\in K$, and put
$b=x-\eta k$. For every $j$,
$\langle f_j,b_n\rangle\to\langle f_j,b\rangle$.
Since $p(b_n)\leqslant1$, the norming property yields $p(b)\leqslant1$.
In fact $p(b)=1$. If $p(b)<1$, then
\[
 x+(1-p(b))\{h:p(h)<1\}\subset B_p+\eta K=B,
\]
so $x$ is an interior point of $B$, contrary to $q(x)=1$.

For each fixed $j$ we have
\[
 \liminf_n p(b_n+b)
 \geqslant\lim_n|\langle f_j,b_n+b\rangle|=2|\langle f_j,b\rangle|.
\]
Taking the supremum over $j$ gives
$\liminf_n p(b_n+b)\geqslant2p(b)=2$.
The reverse upper bound follows from $p(b_n)\leqslant1$.
Local uniform rotundity of $p$ now gives $p(b_n-b)\to0$, and consequently
\[
 x_n=b_n+\eta k_n\weakto b+\eta k=x
\]
along the extracted subsequence. Every subsequence has such a further
subsequence, which proves weak convergence of the original sequence.
Explicitly, failure of weak convergence would give a functional $g$, an
$\varepsilon>0$, and a subsequence with
$|\langle g,x_n-x\rangle|\geqslant\varepsilon$, a contradiction.

Finally, under \eqref{eq:coordinates}, the case $x=0$ follows immediately
from \eqref{eq:equiv}. If $x\neq0$, then $q(x_n)>0$ eventually, and
\[
 \left\langle f_j,\frac{x_n}{q(x_n)}\right\rangle
 \longrightarrow\left\langle f_j,\frac{x}{q(x)}\right\rangle
 \quad(j\in\N).
\]
Apply the unit-sphere case and then multiply by $q(x_n)\to q(x)$.

To prove the smoothness assertion, assume that $p$ is G\^ateaux smooth.
Fix $x\in S_q$ and an exact decomposition $x=b+\eta k$, where
$b\in B_p$ and $k\in K$. The same interior-point argument gives $p(b)=1$.
Let $g\in D_q(x)$. For every $y\in B_p$, we have $y+\eta k\in B$, and
therefore
\[
 \Rea\langle g,y\rangle
 \leqslant1-\eta\Rea\langle g,k\rangle
 =\Rea\langle g,b\rangle.
\]
Taking the supremum over $y\in B_p$ yields
$\Rea\langle g,b\rangle=\|g\|_{p^*}$.
Since $|\langle g,b\rangle|\leqslant\|g\|_{p^*}$, this also gives
$\langle g,b\rangle=\|g\|_{p^*}$ over the complex field. Hence
$g/\|g\|_{p^*}$ is the unique element $g_b$ of $D_p(b)$.
It follows that every element of $D_q(x)$ lies on the positive ray
generated by $g_b$. The condition $\|g\|_{q^*}=1$ then forces
$g=g_b/\|g_b\|_{q^*}$. Thus $D_q(x)$ is a singleton, proving the claim.

Suppose now that $p$ is Fr\'echet smooth, and retain the decomposition
$x=b+\eta k$. Define
$h_K(g)=\sup_{z\in K}\Rea\langle g,z\rangle$. Since $B=B_p+\eta K$,
the dual norms satisfy
\begin{equation}\label{eq:dual-sum}
 \|g\|_{q^*}=\|g\|_{p^*}+\eta h_K(g),\qquad
 0\leqslant h_K(g)\leqslant\|g\|_{p^*}.
\end{equation}
Let $(g_n)\subset S_{q^*}$ satisfy
$\Rea\langle g_n,x\rangle\to1$. The two terms on the right in
\[
 1-\Rea\langle g_n,x\rangle
 =\bigl(\|g_n\|_{p^*}-\Rea\langle g_n,b\rangle\bigr)
  +\eta\bigl(h_K(g_n)-\Rea\langle g_n,k\rangle\bigr)
\]
are non-negative, so the first tends to zero. Put $a_n=\|g_n\|_{p^*}$.
By \eqref{eq:dual-sum}, $(1+\eta)^{-1}\leqslant a_n\leqslant1$.
Consequently $g_n/a_n\in S_{p^*}$ and
$\Rea\langle g_n/a_n,b\rangle\to1$. Fr\'echet smoothness of $p$ and
\v Smulyan's criterion give
$\|g_n/a_n-g_b\|_{p^*}\to0$. The equivalent dual norm $q^*$ then gives
\[
 \frac1{a_n}=\|g_n/a_n\|_{q^*}\longrightarrow\|g_b\|_{q^*}>0.
\]
It follows that $g_n\to g_b/\|g_b\|_{q^*}$ in dual norm. Another
application of \v Smulyan's criterion proves Fr\'echet smoothness of $q$.
\end{proof}

\begin{lemma}\label{lem:restricted-separation}
Let $p$ be an equivalent LUR norm on a real or complex Banach space $X$,
let $K\subset B_p$ be non-empty, absolutely convex and weakly compact,
and let $\eta,\delta>0$. Let $q$ be the Minkowski functional of
$B_p+\eta K$. Suppose that $T:X\to H$ is a bounded linear operator into
a Hilbert space and that its restriction to
$\overline{\operatorname{span}}K$ is injective. Then
$N(x)=(q(x)^2+\delta^2\|Tx\|_H^2)^{1/2}$ is an equivalent wLUR norm.
\end{lemma}
\begin{proof}
The closedness argument in Lemma~\ref{lem:enlargement} uses only weak
compactness of $K$, so it applies here as well. Thus $B_p+\eta K$ is
exactly $B_q$. We first observe that
\begin{equation}\label{eq:radial-estimate}
 q(b+\eta k)\leqslant\frac{p(b)+\eta}{1+\eta}
 \qquad(b\in B_p,\ k\in K).
\end{equation}
Indeed, $(1+\eta)k\in B_q$, so $q(\eta k)\leqslant\eta/(1+\eta)$.
If $a=p(b)>0$, then $b/a+\eta k\in B_q$ and
\[
 b+\eta k=a(b/a+\eta k)+(1-a)\eta k.
\]
The triangle inequality yields \eqref{eq:radial-estimate}; when $a=0$,
the preceding estimate for $q(\eta k)$ suffices.

Suppose $N(x_n)=N(x)=1$ and $N(x_n+x)\to2$. The Hilbert-space vectors
$(q(x_n),\delta Tx_n)$ and $(q(x),\delta Tx)$ have norm one and their
sum has norm tending to two. Hence
$q(x_n)\to a=q(x)>0$ and $Tx_n\to Tx$ in $H$. In addition,
\[
 q(x_n+x)^2=N(x_n+x)^2-\delta^2\|Tx_n+Tx\|_H^2
 \longrightarrow4a^2.
\]
Set $u_n=x_n/q(x_n)$ and $u=x/a$. Then $q(u_n)=q(u)=1$,
$Tu_n\to Tu$, and $q(u_n+u)\to2$. For the last assertion, use
\[
 q\!\left(u_n+u-\frac{x_n+x}{a}\right)
 =\left|\frac1{q(x_n)}-\frac1a\right|q(x_n)\longrightarrow0.
\]

Choose exact decompositions $u_n=b_n+\eta k_n$ and $u=b+\eta k$, with
$b_n,b\in B_p$ and $k_n,k\in K$. By \eqref{eq:radial-estimate},
$q(u)=1$ forces $p(b)=1$. Applying the same estimate to the midpoint
gives
\[
 q(u_n+u)\leqslant\frac{p(b_n+b)+2\eta}{1+\eta}.
\]
Thus $p(b_n+b)\to2$, and LUR gives $p(b_n-b)\to0$. It follows that
$Tk_n\to Tk$. Every subsequence of $(k_n)$ has a further subsequence
converging weakly to some $k'\in K$. Weak continuity of $T$ gives
$Tk'=Tk$, and injectivity on $\overline{\operatorname{span}}K$ implies
$k'=k$. The subsequence principle now gives $k_n\weakto k$, hence
$u_n\weakto u$ and $x_n\weakto x$.
\end{proof}

\section{The construction and its consequences}

\begin{proof}[Proof of Theorem~\ref{thm:approximation}]
Fix $\eta,\delta>0$ and a norming sequence $(f_j)$ for the prescribed
LUR norm $p$, as in \eqref{eq:norming}. Choose $e\in S_p$ and
$f\in X^*$ such that $\|f\|_{p^*}=\langle f,e\rangle=1$.

There is a weakly null sequence $(w_n)$ with $p(w_n)=1$ for all $n$.
Indeed, start with a weakly null sequence which is not norm null,
pass to a subsequence whose norms are bounded away from zero, and
normalise. Put
\[
 v_n=w_n-\langle f,w_n\rangle e.
\]
Then $v_n\in\ker f$, $v_n\weakto0$, and $p(v_n)\to1$.
After discarding finitely many terms, set $z_n=v_n/p(v_n)$. Thus
\begin{equation}\label{eq:zn}
 z_n\weakto0,\qquad p(z_n)=1,\qquad\langle f,z_n\rangle=0.
\end{equation}

Let $K=\overline{\aconv}\{z_n:n\in\N\}$, and let $q$ be the
Minkowski functional of $B_p+\eta K$. Lemmas~\ref{lem:compact} and
\ref{lem:enlargement} apply. Define
\begin{equation}\label{eq:T}
 Tx=(2^{-j/2}\langle f_j,x\rangle)_{j\in\N}\in\ell_2(\F).
\end{equation}
This operator is injective by \eqref{eq:norming}. If $T_m$ retains only
its first $m$ coordinates, then
\begin{equation}\label{eq:T-tail}
 \|T-T_m\|_{(X,p)\to\ell_2}\leqslant2^{-m/2}.
\end{equation}
Thus $T$ is compact, and $\|T\|\leqslant1$. Define
\begin{equation}\label{eq:N}
 N(x)=\bigl(q(x)^2+\delta^2\|Tx\|_2^2\bigr)^{1/2}.
\end{equation}
This is a norm, by the triangle inequality in the Hilbert sum, and
\begin{equation}\label{eq:N-equivalence}
 \frac{p(x)}{1+\eta}\leqslant N(x)
 \leqslant\sqrt{1+\delta^2}\,p(x).
\end{equation}
If $p$ is G\^ateaux smooth, respectively Fr\'echet smooth,
Lemma~\ref{lem:enlargement} gives the same property for $q$.
Both $q^2$ and $\|T\,\cdot\|_2^2$ then have the corresponding
differentiability everywhere, including at zero. Taking the square root
shows that $N$ has the claimed smoothness away from zero.

We show that $N$ is weakly LUR. Suppose
\[
 N(x_n)=N(x)=1,\qquad N(x_n+x)\longrightarrow2.
\]
In the Hilbert space $\F\oplus_2\ell_2(\F)$, put
\[
 h_n=(q(x_n),\delta Tx_n),\qquad h=(q(x),\delta Tx).
\]
Both vectors have norm one, and the triangle inequality for $q$ gives
\[
 N(x_n+x)\leqslant\|h_n+h\|_2\leqslant2.
\]
The parallelogram identity gives
\[
 \|h_n-h\|_2^2=4-\|h_n+h\|_2^2\longrightarrow0.
\]
Consequently $q(x_n)\to q(x)$ and
$\langle f_j,x_n\rangle\to\langle f_j,x\rangle$ for every $j$.
Lemma~\ref{lem:enlargement} now gives $x_n\weakto x$.

It remains to prove that $N$ is not MLUR. Since $K\subset\ker f$,
we have $|\langle f,y\rangle|\leqslant1$ for every $y\in B_p+\eta K$; hence
\begin{equation}\label{eq:f-bound}
 |\langle f,y\rangle|\leqslant q(y)\qquad(y\in X).
\end{equation}
Since $K$ is balanced, $e,e\pm\eta z_n\in B_p+\eta K$.
Together with \eqref{eq:f-bound}, this gives the exact equalities
\begin{equation}\label{eq:flat}
 q(e)=q(e+\eta z_n)=q(e-\eta z_n)=1.
\end{equation}
Compactness of $T$ and weak nullity of $(z_n)$ give $\|Tz_n\|_2\to0$.
Writing $c=N(e)>0$, we obtain
\begin{equation}\label{eq:witness}
 N(e\pm\eta z_n)\longrightarrow c,\qquad
 N(z_n)\geqslant\frac1{1+\eta}.
\end{equation}

For explicit witnesses on the unit sphere, let
\[
 d_n^\pm=N(e\pm\eta z_n),\qquad
 u_n^\pm=\frac{e\pm\eta z_n}{d_n^\pm},\qquad u=\frac ec.
\]
Then $N(u_n^\pm)=N(u)=1$. Moreover, the errors
\[
 E_n^\pm=N\!\left(u_n^\pm-\frac{e\pm\eta z_n}{c}\right)
 =\left|\frac1{d_n^\pm}-\frac1c\right|d_n^\pm
 \longrightarrow0
\]
give
\[
 N\!\left(\frac{u_n^++u_n^-}{2}-u\right)
 \leqslant\frac{E_n^++E_n^-}{2}\longrightarrow0.
\]
On the other hand,
\[
 N(u_n^+-u_n^-)
 \geqslant\frac{2\eta}{c}N(z_n)-E_n^+-E_n^-,
\]
and hence
\[
 \liminf_n N(u_n^+-u_n^-)
 \geqslant\frac{2\eta}{c(1+\eta)}>0.
\]
Thus $N$ is not MLUR. Taking $N_{\eta,\delta}=N$ proves the theorem.
\end{proof}

\begin{proof}[Proof of Theorem~\ref{thm:main}]
If $X$ fails the Schur property, choose an equivalent norm $p$ which is
both LUR and G\^ateaux smooth, using the simultaneous renorming theorem
and, in the complex case, Lemma~\ref{lem:rotation}.
Theorem~\ref{thm:approximation}, with $\eta=\delta=1$, proves (3).
Since LUR implies MLUR, (3) implies (2). Finally, on a space with the
Schur property, the weak convergence supplied by a weakly LUR norm is
norm convergence, so that norm is LUR. This proves (2)$\Rightarrow$(1).
\end{proof}

\begin{corollary}\label{cor:density}
Let $X$ be a separable real or complex Banach space without the Schur
property. For every equivalent norm $r$ and every $\tau>0$, there is an
equivalent wLUR norm $N$ which is not MLUR and satisfies
\[
 \frac{r(x)}{1+\tau}\leqslant N(x)\leqslant(1+\tau)r(x)
 \qquad(x\in X).
\]
In particular, such norms are dense among all equivalent norms for
uniform convergence on bounded sets.
\end{corollary}
\begin{proof}
Choose an equivalent LUR norm $s$ and rescale it so that $s\leqslant r$.
Put $p_\tau(x)=(r(x)^2+\tau s(x)^2)^{1/2}$. Then
$r\leqslant p_\tau\leqslant\sqrt{1+\tau}\,r$.
We first check that $p_\tau$ is LUR. If
$p_\tau(x_n)=p_\tau(x)=1$ and $p_\tau(x_n+x)\to2$, the Hilbert vectors
$a_n=(r(x_n),\sqrt\tau s(x_n))$ and
$a=(r(x),\sqrt\tau s(x))$ have norm one, and
$p_\tau(x_n+x)\leqslant\|a_n+a\|_2\leqslant2$. Hence $a_n\to a$.
In particular, $s(x_n)\to s(x)>0$. Moreover,
\[
 \begin{split}
 0&\leqslant\tau\bigl((s(x_n)+s(x))^2-s(x_n+x)^2\bigr)\\
  &\leqslant\|a_n+a\|_2^2-p_\tau(x_n+x)^2\longrightarrow0.
 \end{split}
\]
Thus $s(x_n+x)\to2s(x)$. After normalising in $s$, local uniform
rotundity gives norm convergence of $x_n$ to $x$.

Apply Theorem~\ref{thm:approximation} to $p_\tau$ with $\eta=\tau$
and $\delta=\sqrt\tau$. The resulting norm satisfies
$r/(1+\tau)\leqslant N\leqslant(1+\tau)r$, as required.
In particular, $\sup_{x\in B_r}|N(x)-r(x)|\leqslant\tau$.
No smoothness assertion is made here for an arbitrary starting norm $r$.
\end{proof}

\begin{proof}[Proof of Theorem~\ref{thm:general}]
Suppose that $X$ has an equivalent wLUR norm which is not MLUR.
By \cite[Main Theorem]{MOTV}, its underlying real space admits an
equivalent LUR norm. In the complex case,
Lemma~\ref{lem:rotation} makes this a complex LUR norm.
If $X$ had the Schur property, its given wLUR norm would be LUR and
hence MLUR. Thus (1) implies (2).

Conversely, fix an equivalent LUR norm $p$ on $X$ and suppose that $X$
fails the Schur property. As in the proof of
Theorem~\ref{thm:approximation}, choose $e\in S_p$,
$f\in S_{p^*}$ with $\langle f,e\rangle=1$, and a normalised weakly
null sequence $(z_n)\subset\ker f$. Set
$K=\overline{\aconv}\{z_n:n\in\N\}$ and
$Y=\overline{\operatorname{span}}\{z_n:n\in\N\}$.
The set $K\subset B_p$ is weakly compact by Lemma~\ref{lem:compact},
and $Y$ is separable.

Choose a countable norming family $(g_j)$ for $(Y,p)$ as in
\eqref{eq:norming}, and extend it by the Hahn--Banach theorem to
$f_j\in X^*$ with $\|f_j\|_{p^*}=1$. The operator
$Tx=(2^{-j/2}\langle f_j,x\rangle)_{j\in\N}$ is compact by
\eqref{eq:T-tail}, and $T|_Y$ is injective. Let $q$ be the Minkowski
functional of $B_p+K$ and put
$N(x)=(q(x)^2+\|Tx\|_2^2)^{1/2}$.
Lemma~\ref{lem:restricted-separation} shows that $N$ is wLUR.

The remaining estimates do not require the family $(f_j)$ to norm all
of $X$. In fact, $K\subset\ker f$ gives
$q(e)=q(e\pm z_n)=1$, and compactness gives $Tz_n\to0$.
Hence, with $c=N(e)$,
$N(e\pm z_n)\to c$ and $N(z_n)\geqslant1/2$.
The normalised vectors $u_n^\pm=(e\pm z_n)/N(e\pm z_n)$ therefore
satisfy, by the same normalisation estimates used in the proof of
Theorem~\ref{thm:approximation},
\[
 N\!\left(\frac{u_n^++u_n^-}{2}-\frac ec\right)\longrightarrow0,
 \qquad \liminf_nN(u_n^+-u_n^-)\geqslant\frac1c>0.
\]
Thus $N$ is not MLUR, proving (2)$\Rightarrow$(1).
\end{proof}

\begin{remark}
The shrinking assumption in the coordinate construction of the first-named
author \cite{Draga} cannot simply be replaced by totality of the
coordinate functionals. To see this, apply that construction to the
canonical Markushevich basis of $\ell_1$. The resulting norm is
\[
 r(x)=\max\{\tfrac12\|x\|_1,\|x\|_\infty\},\qquad
 M(x)^2=r(x)^2+\sum_{j=1}^{\infty}4^{-j}|x_j|^2.
\]
For $n\geqslant2$,
\[
 M(e_1+e_n)\longrightarrow\frac{\sqrt5}{2}=M(e_1),\qquad
 M(2e_1+e_n)\longrightarrow\sqrt5,
\]
whereas $\langle\mathbf1,e_n\rangle=1$ for
$\mathbf1=(1,1,\ldots)\in\ell_\infty=\ell_1^*$.
Put $c=\sqrt5/2$, $y_n=(e_1+e_n)/M(e_1+e_n)$, and $y=e_1/c$.
Then $M(y_n)=M(y)=1$ and $M(y_n+y)\to2$, while
\[
 \langle\mathbf1,y_n\rangle\longrightarrow\frac2c
 \neq\frac1c=\langle\mathbf1,y\rangle.
\]
Hence $M$ is not weakly LUR.
\end{remark}

\end{document}